\documentclass[11pt]{article}
\usepackage {graphics}
\usepackage {pstricks,pst-node} 
\usepackage{mathtools}
\usepackage{amsmath,amssymb,amsbsy}
\usepackage{multirow}
\usepackage{amsfonts}
\usepackage{amsthm}
\usepackage{graphicx, enumerate}
\usepackage{latexsym}
\usepackage{longtable}
\usepackage{tabularx}
\usepackage{setspace}
\usepackage{float}
\usepackage{rotating}
\usepackage{tikz}
\usepackage{verbatim}
\usepackage{soul}
\usepackage{oldlfont,dsfont}
\usetikzlibrary{graphs}
\usepackage[T1]{fontenc}
\usepackage[polish, english]{babel}
\usepackage{polski}

\newtheorem{theorem}{Theorem}
\newtheorem{conjecture}[theorem]{Conjecture}
\newtheorem{corollary}[theorem]{Corollary}
\newtheorem{proposition}[theorem]{Proposition}

\newtheorem{observation}[theorem]{Observation}

\newtheorem{remark}[theorem]{Remark}

\title{Strong majority colorings of graphs}
\author{Rafa{\l} Kalinowski, Mateusz Kamyczura, \\ Monika Pil\'sniak, Mariusz Wo{\'z}niak \\\mbox{ }\\
AGH University of Krakow,  Poland
}
\date{}

\begin{document}
\maketitle

\begin{abstract}
Motivated by majority vertex-colorings of graphs and digraphs and majority edge-colorings of graphs, we introduce two concepts of strong majority colorings. A strong majority vertex-coloring of a graph $G=(V,E)$ is a mapping $c:V\rightarrow C$ such that for every vertex $v\in V$ and every color $\alpha\in C$, at most half of the neighbors of $v$ have color $\alpha$. The strong majority number of $G$, denoted Maj$(G)$, is the least number of colors in such a coloring. We show that Maj$(G)$ can be arbitrarily large and prove a tight upper bound Maj$(G)\le 2\Delta(G)+1$ for every graph $G$ without pendant vertices. A strong majority edge-coloring of a graph $G$ is a mapping $c:E\rightarrow C$ such that for every edge $e\in E$ and every color $\alpha\in C$, at most half of the edges adjacent to $e$ have color $\alpha$. The strong majority index of $G$, denoted Maj'$(G)$, is the least number of colors in such a coloring. It is shown that there is an upper constant bound for Maj'$(G)$ of all admissible graphs $G$. We conjecture that this constant is as small as 4 and confirm this conjecture for numerous graph classes.
\end{abstract}

\section{Motivation}

We introduce two graph invariants related to sizes of color classes in the neighborhood of a vertex or an edge. To make this paper more readable, we also give names and notation to analogous invariants previously considered in the literature, which were an inspiration to our research. 

Let $G$ be a simple undirected graph. We call a vertex-coloring of $G$ {\it majority} if for every vertex $v$ at most half its neighbors share the color of $v$. By ${\mathrm maj}(G)$ we denote the minimum number of colors among all majority vertex-colorings of the graph $G$. These colorings were first considered in 1966 by Lov\'asz \cite{Lov66}, who proved the following classical result.

\begin{theorem} [Lov\'asz \cite{Lov66}]\label{Lov}
For every finite graph $G$, $${\mathrm maj}(G)=2.$$   
\end{theorem}

Majority vertex-colorings were later investigated for infinite graphs. In this case, for a subset set $S$ of an infinite set $\Omega$, the condition that the cardinality of $S$ is at most half the elements of $\Omega$  means that the complement of $S$ is of the same cardinality as $\Omega$. Shelah and Milner \cite{ShMi90} showed that ${\mathrm maj}(G)\le 3$ for every graph of arbitrary cardinality, and there exist uncountable graphs with ${\mathrm maj}(G)=3$. The renowned Unfriendly Partition Conjecture states that ${\mathrm maj}(G)=2$ for every countable graph $G$. This conjecture is often viewed as one of the most intriguing and difficult open problems in infinite graph theory. It has been confirmed only for some classes of graphs, in particular for \\ 
-- graphs with finitely many vertices of infinite degree~(Aharoni, Milner, and Prikry \cite{AMP90}),
\\ -- graphs with finitely many vertices of finite degree~(Aharoni, Milner, and Prikry \cite{AMP90}),
\\ -- rayless graphs~(Bruhn, Diestel, Georgakopoulos, and Spr\"ussel \cite{BDGS10}),
\\ -- graphs without a subdivision of an infinite clique~(Berger \cite{Ber17}),
\\ -- line graphs (Kalinowski, Pil\'sniak, and Stawiski \cite{KPS25}).

\vspace{6mm}
In 2017, Kreutzer, Oum, Seymour, van der Zyphen, and Woods \cite{KOSvW17} introduced the following analogous notion for digraphs. A vertex-coloring of a digraph $D$ is {\it majority} if for every vertex $v$ at most half its out-neighbors share the color of $v$. Let us denote the minimum number of colors in such a coloring by $\overrightarrow{{\mathrm maj}}(D)$. In \cite{KOSvW17} it was shown that $\overrightarrow{{\mathrm maj}}(D)\le 4$, and it was conjectured that $\overrightarrow{{\mathrm maj}}(D)\le 3$ for every finite digraph $D$. This conjecture is still open but confirmed in some cases (cf. \cite{ABGGPZ25} and references therein).

\vspace{6mm}

In 2023, Bock, Kalinowski, Pardey, Pil\'sniak, and Wo\'zniak \cite{BKPPRW23} defined {\it majority edge-colorings} of graphs. These are edge-colorings such that for every vertex $v$ and every color $\alpha$, at most half of the edges incident with $v$ have the color $\alpha$. Clearly, a graph having a vertex of degree 1 does not admit such a coloring. Let us denote the minimum number of colors in a majority edge-coloring of a graph $G$ by ${\mathrm maj'\,}(G)$. 

The following upper bound for ${\mathrm maj'\,}(G)$ was established in \cite{BKPPRW23}.

\begin{theorem} [\cite{BKPPRW23}]
Every finite graph $G$ with minimum degree $\delta(G)\ge 2$ satisfies  $${\mathrm maj'\,}(G)\le 4.$$
Moreover, if $\delta(G)\ge 4$, then ${\mathrm maj'\,}(G)\le 3.$
\end{theorem}
The first part of the above theorem was extended by Kalinowski, Pil\'sniak, and Stawiski in \cite{KPS25} to infinite graphs of arbitrary cardinality, also for the list setting. Moreover, as a consequence of their investigations of majority edge-colorings, they proved the following result for majority vertex-colorings.
\begin{theorem} [\cite{KPS25}]\label{line}
If $G$ is an infinite line graph of arbitrary cardinality, then $${\mathrm maj}(G)=2.$$    
\end{theorem}
Thus, using the concept of majority edge-colorings, they confirmed the Unfriendly Partition Conjecture for line graphs. 

\vspace{5mm} 
\begin{remark}
{\rm Suppose that for an edge-coloring of a graph $G=(V,E)$ we require that, for every edge $e\in E$, at most half of the edges adjacent to $e$ have the color of $e$ (analogously to a majority vertex-coloring). Then, this coloring is equivalent to  majority vertex-colorings of the line graph of $G$. Consequently,  due to Theorem {\rm \ref{Lov}} and Theorem {\rm \ref{line}}, two colors suffice for such an edge-coloring of arbitrary cardinality. } 
\end{remark}

\section{Strong majority vertex-colorings}

Throughout, graph colorings need not be proper, and graphs are finite  and undirected, unless otherwise stated. 
Let $c$ be a vertex-coloring of a graph $G$. If for every vertex $v$ and every color $\alpha$, at most half at most half of the neighbors of $v$ have color $\alpha$, then $c$ is a~ {\it strong majority vertex-coloring} of $G$. The minimum number of colors in such a coloring is the {\it strong majority number} of $G$, denoted by ${\mathrm Maj}(G)$. Clearly, ${\mathrm Maj}(G)$ is defined only for graphs $G$ without pendant vertices, that is, for graphs $G$ with $\delta(G)\ge 2$.

Let us start with two easy examples.
\begin{observation}\label{cycle}
For a cycle $C_n$, we have ${\mathrm Maj}(C_n)=2$ if $n\equiv 0 \mod{4}$, and ${\mathrm Maj}(C_n)=3$ otherwise.  
\end{observation}
\begin{proof}
In a strong majority vertex-coloring of $C_n$, two vertices of distance 2 necessarily have two distinct colors. Color subsequent vertices with colors $1,1,2,2,\ldots$. This yields a strong majority vertex-coloring when $n\equiv 0 \mod{4}$. Otherwise, a third color is needed.   If $n\equiv 1 \mod{4}$, then we put 3 at the end of the above sequence of  colors. If $n\equiv 2\mod 4$, we end with $3,3$. For $n\equiv 3\mod 4$, we put $1,3,3$ on the last three vertices.
\end{proof}
\begin{observation}
For a complete graph $K_n$ with $n\ge 3$, ${\mathrm Maj}(K_n)=3$ except for $n=4$, when ${\mathrm Maj}(K_n)=4$.   
\end{observation}
\begin{proof}
Partition $V(K_n)$ into three subsets, as equitable as possible. Choose a different color for each subset. The equality $\lceil\frac n3\rceil\le \frac{n-1}{2}$ holds whenever $n\ne 4$. For $n=4$, if two vertices had the same color, then every other vertex had $2/3$ of neighbors of the same color.  
\end{proof}

The following two results immediately follow from the paper \cite{KOSvW17} of Kreutzer, Oum, Seymour, van der Zyphen, and Woods.  

\begin{theorem}[\cite{KOSvW17}, Theorem 3]
Let $G$ be a graph of order $n$. If $\delta(G)> 72\ln 3n$,
then $${\mathrm Maj}(G)\le 3.$$
\end{theorem}
\begin{theorem} [\cite{KOSvW17}, Theorem 4]
Let $G$ be a graph with minimum degree $\delta\ge 1200$ and maximum degree at most $\exp{(\delta/72)}/12\delta$. Then $${\mathrm Maj}(G)\le 3.$$
\end{theorem} 

In Section \ref{edge}, we show that there exists a constant $C$ such that ${\mathrm Maj}(G)\le C$ for every line graph $G$ with $\delta(G)\ge 2$.

\vspace{0.5cm}

However, there are graphs with arbitrarily large ${\mathrm Maj}(G)$. Let $\widehat{K_n}$ be a~subdivision of the complete graph $K_n$, that is, a graph obtained from $K_n$ by replacing each edge by a path of length two. Clearly, ${\mathrm Maj}(\widehat{K_n})=n$ since each vertex $v\in V(K_n)$ has to have a distinct color because neighbors of a~vertex of degree two must have distinct colors. 
 Another example is a graph $\widetilde{K_n}$ obtained from $K_n$ by replacing each edge $uv$ by a diamond $K_4-e$ such that $u$ and $v$ are of degree 2 in $K_4-e$ (see Figure \ref{Kn}). Again, ${\mathrm Maj}(\widetilde{K_n})=n$.

%%%%%%%%%%%%%%%%%%%%%%

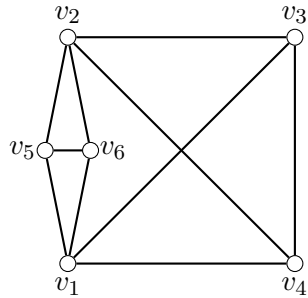
\begin{figure}[htb]
    \centering
    
\begin{tikzpicture}[scale=3]
    % Define styles
    \tikzstyle{vertex}=[circle, draw, inner sep=0pt, minimum size=6pt]
    \tikzstyle{plugger}=[circle, draw, inner sep=0pt, minimum size=6pt]
    
    % Draw vertices
    \node[vertex] (1) at (0,0) {};
    \node[vertex] (2) at (0,1) {};
    \node[vertex] (3) at (1,1) {};
    \node[vertex] (4) at (1,0) {};
    \node[plugger] (5) at (-0.1,0.5) {};
    \node[plugger] (6) at (0.1,0.5) {};
    
    % Draw edges
    \draw[thick] (1) -- (4) -- (3) -- (2) -- (4);
    \draw[thick] (1) -- (3);  % New edge 1-3
    
    % Draw plugger edges
    \draw[thick] (1) -- (5) -- (2);
    \draw[thick] (1) -- (6) -- (2);
    \draw[thick] (5) -- (6);
    
    % Label vertices
    \node at (0,-0.1) {$v_1$};
    \node at (0,1.1) {$v_2$};
    \node at (1,1.1) {$v_3$};
    \node at (1,-0.1) {$v_4$};
    \node at (-0.2,0.5) {$v_5$};
    \node at (0.2,0.5) {$v_6$};
\end{tikzpicture}

    \caption{Edge $v_1v_2$ of $K_4$ replaced by a diamond}
    \label{Kn}
\end{figure}

%%%%%%%%%%%%%%%%%%%%%%%%tutu

The above two examples have minimum degree 2 and 3, respectively. However, graphs with large strong majority numbers may have arbitrarily large minimum degree, as the following result shows.

\begin{theorem}
For any two integers $K$ and $\delta\ge 2$, there exists a graph $G$ with $\delta(G)=\delta$ and ${\mathrm Maj}(G)\ge K$.
\end{theorem}

\begin {proof}
We construct a bipartite graph  $G = (X \cup Y, E)$  with  $|X| = \lfloor \frac {\delta}{2} \rfloor K$ and  $|Y| = {\lfloor \frac {\delta}{2} \rfloor K \choose \delta}$. There is a bijection $g:Y\rightarrow {X\choose \delta}$. For every vertex $y\in Y$, we put edges between
$y$ and all vertices of $g(y)$. Thus, the degree of any vertex $x\in X$ is equal to $d(x)={|X|-1\choose \delta-1}$. Moreover, $N_G(y)=g(y)$, so $\delta(G)=d(y)=\delta$, for every $y\in Y$.
In a strong majority coloring of $G$, at most $\lfloor \frac{\delta}{2} \rfloor$ vertices of $X$ can have the same color. Therefore, at least $K$ colors are needed. 
\end{proof}

\vspace{4mm}
There is the following upper bound for ${\mathrm Maj}(G)$ of a given graph $G$ .
\begin{theorem}\label{2Delta+1}
For any graph $G$ with $\delta(G)\ge 2$, $${\mathrm Maj}(G) \le 2\Delta(G)+1.$$
Moreover, for infinitely many integers $\Delta$ there exist graphs $G$ with $\Delta(G)=\Delta$ and ${\mathrm Maj}(G) = 2\Delta(G)+1.$
\end{theorem}
\begin{proof}
Let $H$ be an auxiliary graph with $V(H)=V(G)$ obtained as follows:  For each vertex $v\in V(G)$, if $d_G(v)\ge 3$, introduce in $H$ a cycle on the neighbors of $v$ in $G$ (ordered arbitrarily), and if $d_G(v)=2$, introduce an edge between the two neighbors of $v$.  Note that $\Delta(H)\le 2\Delta(G)$ and that every proper coloring of vertices of $H$ is a strong majority vertex-coloring of $G$. By Brooks' theorem, $\Delta(H)+1$ colors suffice for this coloring.

To justify the second part of the claim, we construct an infinite family of graphs as follows. Take any $n\equiv 1, 3 \mod{6}$. There exists a Steiner triple system STS$(n)$. Let $G$ be the {\it block-point incidence graph of } STS$(n)$, that is, a~bipartite graph $G=(X\cup Y,E)$, where $X=\{x_1,\ldots,x_s\}$ with $s=\frac 16 n(n-1)$, is the set of triples in STS$(n)$, $Y=\{y_1,\ldots,y_n\}$ is the set of points in STS$(n)$, and $E=\{xy: x\in X,y\in Y, y\in x\}$ (the case $n=7$ is illustrated in Figure \ref{STS(3)}). Each vertex in $Y$ has degree $(n-1)/ 2$, while each vertex of $X$ has degree 3. Thus, $\Delta(G)=(n-1)/ 2$. From the definition of STS$(n)$ it follows that for each pair of vertices $y_i,y_j\in Y$, there exists a vertex $y_k\in Y$ such that the triple $\{y_i,y_j,y_k\}$ is a vertex in $X$. Consequently, every vertex of $Y$ has to get a distinct color in a majority vertex-coloring of $G$. Hence, Maj$(G)=n=2\Delta(G)+1$. 

As $\Delta(G)=(n-1)/ 2$, our construction shows that for every positive integer $\Delta\ge 3$, such that $\Delta\equiv 0,1,3,4 \mod 6$, there exists a graph $G$ with $\Delta(G)=\Delta$ and Maj$(G)=2\Delta(G)+1$.
\end{proof}

%+++++++++++++++++++++++++++++++++++++++++++++++++++++++++++++++++++++

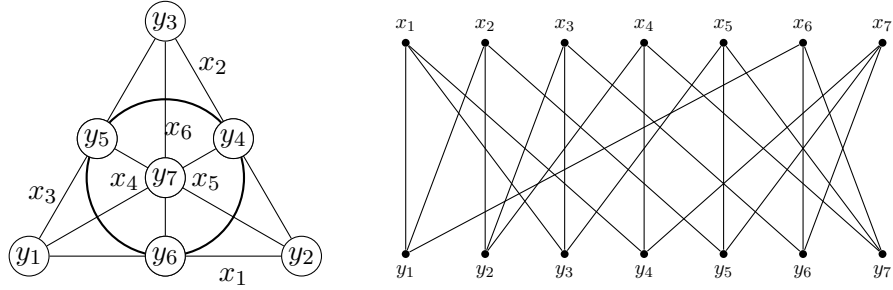
\begin{figure}[htb]
\centering
\begin{minipage}{0.45\textwidth}
    \centering
    \begin{tikzpicture}[scale=0.9, every node/.style={circle, draw, fill=white, inner sep=1pt}]
        % Wierzchołki dla pierwszego grafu (Fano plane)
\node (A1) at (0, 0) {$y_1$};
\node (A2) at (4, 0) {$y_2$};
\node (A3) at (2, 3.46) {$y_3$};
\node (A4) at (3, 1.73) {$y_4$};
\node (A5) at (1, 1.73) {$y_5$};
\node (A6) at (2, 0) {$y_6$};
\node (A7) at (2, 1.155) {$y_7$};

% Krawędzie Fano plane
\draw (A1) -- (A2);
\draw (A2) -- (A3);
\draw (A3) -- (A1);
\draw (A1) -- (A4);
\draw (A2) -- (A5);
\draw (A3) -- (A6);

% Okrąg w Fano plane
\draw[thick] (2, 1.155) circle (1.155);

% Podpisy krawędzi
\node[draw=none,fill=none] at (3, -0.3) {$x_1$};
\node[draw=none,fill=none] at (2.7, 2.8) {$x_2$};
\node[draw=none,fill=none] at (0.2, 0.9) {$x_3$};
\node[draw=none,fill=none] at (1.4, 1.1) {$x_4$};
\node[draw=none,fill=none] at (2.6, 1.1) {$x_5$};
\node[draw=none,fill=none] at (2.2, 1.85) {$x_6$};

\node (A4) at (3, 1.73)  [circle, draw, fill=white, inner sep=1pt] {$y_4$};
\node (A5) at (1, 1.73) [circle, draw, fill=white, inner sep=1pt] {$y_5$};
\node (A6) at (2, 0) [circle, draw, fill=white, inner sep=1pt] {$y_6$};
\node (A7) at (2, 1.155) [circle, draw, fill=white, inner sep=1pt] {$y_7$};

    \end{tikzpicture}
\end{minipage}%
\hfill
\begin{minipage}{0.55\textwidth}
    \centering
    \begin{tikzpicture}[scale=0.7, every node/.style={transform shape}]
        % --- Bipartite graph ---
        \node[fill=black, circle, inner sep=1.5pt, label=above:{$x_1$}] (x1) at (6, 4) {};
        \node[fill=black, circle, inner sep=1.5pt, label=above:{$x_2$}] (x2) at (7.5, 4) {};
        \node[fill=black, circle, inner sep=1.5pt, label=above:{$x_3$}] (x3) at (9, 4) {};
        \node[fill=black, circle, inner sep=1.5pt, label=above:{$x_4$}] (x4) at (10.5, 4) {};
        \node[fill=black, circle, inner sep=1.5pt, label=above:{$x_5$}] (x5) at (12, 4) {};
        \node[fill=black, circle, inner sep=1.5pt, label=above:{$x_6$}] (x6) at (13.5, 4) {};
        \node[fill=black, circle, inner sep=1.5pt, label=above:{$x_7$}] (x7) at (15, 4) {};

        \node[fill=black, circle, inner sep=1.5pt, label=below:{$y_1$}] (y1) at (6, 0) {};
        \node[fill=black, circle, inner sep=1.5pt, label=below:{$y_2$}] (y2) at (7.5, 0) {};
        \node[fill=black, circle, inner sep=1.5pt, label=below:{$y_3$}] (y3) at (9, 0) {};
        \node[fill=black, circle, inner sep=1.5pt, label=below:{$y_4$}] (y4) at (10.5, 0) {};
        \node[fill=black, circle, inner sep=1.5pt, label=below:{$y_5$}] (y5) at (12, 0) {};
        \node[fill=black, circle, inner sep=1.5pt, label=below:{$y_6$}] (y6) at (13.5, 0) {};
        \node[fill=black, circle, inner sep=1.5pt, label=below:{$y_7$}] (y7) at (15, 0) {};

        % --- Krawędzie ---
        \draw (x1) -- (y1);
        \draw (x1) -- (y3);
        \draw (x1) -- (y4);

        \draw (x2) -- (y1);
        \draw (x2) -- (y2);
        \draw (x2) -- (y5);

        \draw (x3) -- (y2);
        \draw (x3) -- (y3);
        \draw (x3) -- (y6);

        \draw (x4) -- (y2);
        \draw (x4) -- (y4);
        \draw (x4) -- (y7);

        \draw (x5) -- (y3);
        \draw (x5) -- (y5);
        \draw (x5) -- (y7);

        \draw (x6) -- (y1);
        \draw (x6) -- (y6);
        \draw (x6) -- (y7);

        \draw (x7) -- (y4);
        \draw (x7) -- (y5);
        \draw (x7) -- (y6);
    \end{tikzpicture}
\end{minipage}

\caption{STS$(7)$ represented by the Fano plane (where each line $x_i$ is a triple of points) and the block-point incidence graph $G$ of STS$(7)$}
\label{STS(3)}
\end{figure}

\section{Strong majority edge-colorings}\label{edge}

In this section, we introduce an analog of the strong majority vertex-colorings for edge-colorings.

Given an edge-coloring $c$ of a graph $G=(V,E)$, we say that an edge $e\in E$ is {\it majority colored} if there is no color $\alpha$ such that $e$ is adjacent to more edges of color $\alpha$ than to the other edges. The coloring $c$ 
is called a {\it strong majority edge-coloring} if every edge of $G$ is majority colored. Clearly, such a coloring exists only if $G$ does not contain a pendant path of length more than one, that is, there does not exist a vertex of degree one with a neighbor of degree two. We call such graphs {\it admissible}. On the other hand, if a graph $G$ is admissible, then assigning every edge a distinct color yields a strong majority coloring. Hence, from now on in this section we exclusively consider admissible graphs. The least number of colors in a strong majority edge-coloring of an admissible graph $G$ is called the {\it strong majority index} of $G$, denoted by ${\mathrm Maj'}(G)$. 

Observe that Maj'$(G)$ can be equal to 2 only if the degrees of all vertices of $G$ are even. Let us start with two simple examples. The first one immediately follows from Observation \ref{cycle} and the fact that $L(C_n)=C_n$.
\begin{observation}\label{cycle2}
For a cycle $C_n$, ${\mathrm Maj'}(C_n)=2$ if $n\equiv 0 \mod{4}$, and ${\mathrm Maj'}(C_n)=3$ otherwise.  
\end{observation}

Unlike the strong majority number ${\mathrm Maj}(G)$, the strong majority chromatic index ${\mathrm Maj'}(G)$ for all admissible graphs is bounded from above by a constant.

\begin{theorem}\label{edge col}
For every admissible graph $G$, $${\mathrm Maj'}(G)\le 8.$$    
\end{theorem}

\begin{proof} Let $G=(V,E)$ be an admissible graph, and let $C$ be a set of eight colors. We want to show that there is a strong majority edge-coloring of $G$ with  colors of $C$. We may assume that $G$ is connected; otherwise we can argue component-wise.

We construct a graph $G^*$ using the following operation of splitting vertices. If $\Delta(G)\le 3$, then we put  $G^*=G$.  Otherwise, for every vertex $v\in V$  of degree greater than 3, we partition its neighborhood $N(v)$ into subsets $N_1(v),\ldots,N_s(v)$ such that $|N_i(v)|\in \{2,3\}, i=1,\ldots,s$, and the number of $i$'s with $|N_i(v)|=2$ is at most two. Next, we substitute the vertex $v$ by $s$ new vertices $v_1,\ldots,v_s$, and insert an edge between $v_i$ with every vertex of $N_i(v)$.  

Note that $1\le \delta(G^*)\le\Delta(G^*)\leq 3$ and there is a natural bijection between the edges of $G$ and those of $G^*$.  Obviously, $G^*$ need not be connected. Let $H^*$ be a connected component of $G^*$. Now, we construct a proper and strong majority edge-coloring with 8 colors. 

Suppose first that all vertices of $H^*$ have degree 2 in $H^*$, that is, $H^*$ is a cycle $C_m$. It is easy to find a proper majority 4-edge-coloring in this case. Namely, take four colors, say 1,2,3,4, and color the consecutive edges of $C_m$ with three colors periodically: $1,2,3,1,2,3\ldots$. If $m\equiv 0\mod 3$, then we are done. If $m\equiv 1 \mod 3$, then we put color 4 on the last edge of $C_m$. Otherwise, for $m\equiv 2 \mod 3$, we color the last five  edges of $C_m$ with $1,4,2,3,4$.   
 
Let then $\Delta(H^*)=3$. By Vizing's theorem, $G^*$ has a proper edge-coloring $c$ with four colors. 

Let $e=xy$ be any edge of $H^*$. If both vertices $x,y$ are of odd degree, 1 or 3, then the edge $e$ is majority colored by $c$, for $c$ is proper. Consequently, if $e$ is not majority colored in $H^*$, then it is incident to a vertex of degree 2. 

Consider a maximal induced path $P=x_0,\ldots,x_r$ in $H^*$ with $r\ge 2$. Denote by $H^*_P$ the subgraph of $H^*$ induced by the vertices of $P$ and possible neighbors of the end-vertices $x_0$ and $x_r$. If $H^*_P$ contains an edge that is not majority colored, then we recolor some edges of $P$ as follows.

Suppose first that $P$ is a pendant path in $H^*$ and $d_{H^*}(x_r)=1$. If $d_{H^*}(x_0)=1$, then $H^*_P$ coincides with the path $P$. For any set $C'\subset C$ with $|C'|=3$, there exists a proper coloring $c':E(P)\rightarrow C'$ such that all  inner edges of $P$ are majority colored. 
Let then $d_{H^*}(x_0)=3$, and let $u_1,u_2$ be the two neighbors of $x_0$ outside of $P$. If $x_0u_1$ or $x_0u_2$ is not majority colored, then we recolor the edge $x_0x_1$ with a color from $C$ distinct from the colors of the six edges incident to $u_1$ and $u_2$. Then it is easy to see that we have at least $8-3=5$ free colors for the edge $x_1x_2$. Then we recolor subsequent edges $x_2x_3,\ldots, x_{r-1}x_r$ having 6 choices for each. It follows that any pair of free colors out of five can be used for the subpath $x_0x_1x_2$. We will use this fact later.

Now, assume that both end-vertices $x_0,x_r$ are of degree 3 in $H^*$. Denote by $v_1,v_2$ the two neighbors of $x_r$ outside $P$. We may need to recolor some edges of $P$ to obtain a proper coloring of $H^*_P$
such that every edge of $H^*_P$ is majority colored. 

Let $r=2$, and suppose that the edge $x_0x_1$ is not majority colored. It follows that the color $c(x_0x_1)$ also appears on one of the edges $x_0u_1$ or $x_0u_2$. To make the edge $x_0x_1$ majority colored, we recolor the edge $x_1x_2$. To do this, we cannot use the colors $c(x_0u_1), c(x_0u_2), c(x_0x_1),c(x_2v_1), c(x_2v_2)$. Moreover, if the vertices $v_1,v_2$ are of degree 2, then we cannot use the two colors of the edges outside $H^*_P$ incident to them. In total, we may have seven forbidden colors for the edge $x_1x_2$, so we may be forced to use an eighth color. By symmetry, the same arguments are used to show that we can recolor the edge $x_0x_1$ to make the edge $x_1x_2$ majority colored. 

Now, suppose that $r= 3$. If  $x_0x_1$ or $x_2x_3$ is not majority colored, then we can recolor the edge $x_1x_2$ with any color different from $c(x_0u_1), c(x_0u_2), c(x_0x_1)$ and from $c(x_2x_3),c(x_3v_1),c(x_3v_2)$. Then the subgraph $H^*_P$ gets a proper and strong majority coloring.  

Let then $r\ge 4$. If the edge $x_0x_1$ is not majority colored, we first recolor the edge $x_1x_2$. We choose  any color except $c(x_0u_1), c(x_0u_2), c(x_0x_1)$ for the edge $x_1x_2$. Analogously, if the edge $x_{r-1}x_r$ is not majority colored, we recolor the edge $x_{r-2}x_{r-1}$ with any color except  $c(x_{r-1}x_r),c(x_rv_1), c(x_rv_2)$. Then it is easy to see how to recolor the edges $x_2x_3,\ldots,x_{r-3}x_{r-2}$ with three colors from $C\setminus\{c(x_1x_2),c(x_{r-2}x_{r-1})\}$ to obtain a proper and strong majority coloring of $H^*_P$. 

Thus, we obtain a proper edge-coloring of every connected component of $G^*$, where every edge is majority colored except, possibly, some pendant edges of $G$. The sum of these colorings yields a strong majority edge-coloring of the graph $G$. This is because each coloring is proper, so a new component can add two or three edges with different colors incident to an end-vertex of a given edge. If an edge $e$ is not majority colored, then $e=x_0x_1$ belonged in $H^*$ to a pendant path $x_0,x_1,\ldots x_r$ of length $r\ge 2$. As we have noticed above, we can recolor the edges $x_0x_1, x_1x_2$ so that $e=x_0x_1$ becomes majority colored.
\end{proof}

If $G$ is a line graph of an admissible graph $H$, then obviously ${\mathrm Maj}(G)={\mathrm Maj'\,}(H)$. 

\begin{corollary}
If $G$ is a line graph with $\delta(G)\ge 2$, then $${\mathrm Maj}(G)\le 8.$$ 
\end{corollary}

However, we believe that the upper bound in Theorem~\ref{edge col} can be substantially reduced.
\begin{conjecture}\label{Maj'4}
If $G$ is an admissible graph, then 
$${\mathrm Maj'}(G)\le 4.$$ 
\end{conjecture}

The bound 4 would be best possible because there exist graphs $G$ with Maj'$(G)=4$. Let $\widehat{S}$ be a subdivision of a snark $S$ which is a cubic graph of Class 2, that is, with the chromatic index equal to 4. Suppose Maj'$(\widehat{S})\le 3$. Every edge of $\widehat{S}$ has three adjacent edges, so they have three distinct colors in a strong majority edge-coloring. For the same reason, the three edges incident to a vertex of degree 3 must have three distinct colors. Let $x$ be any vertex of degree 2 in $\widehat{S}$ adjacent to vertices $v_1,v_2$ of degree 3. Without loss of generality, assume that the edge $v_1x$ has color 1. Hence, the other edges incident to $v_1$ have colors $2,3$. It follows that the edge $xv_2$ has color 1, the same as $v_1x$. Thus, if  we replace each induced path of length two in $\widehat{S}$ with an edge, we obtain a proper 3-edge-coloring of the snark $S$, contrary to the definition of a snark.

Conjecture \ref{Maj'4} holds for cycles by Observation \ref{cycle2}.  It also holds for complete graphs as Maj'$(K_n)\le 3$ for $n\ge 3$. Indeed, it is easy to verify the claim for $n\le5$. Let $n\ge 6$. It is well known that a complete graph $K_n$ admits a decomposition into Hamiltonian paths or Hamiltonian cycles, depending on the parity of $n$. Partition these paths, respectively, cycles, into three parts as equitable as possible. Assign each part one of three colors. It is not hard to see that this yields a strong majority edge-coloring of $K_n$. 

However, there are much larger classes of graphs for which Conjecture \ref{Maj'4} is true. 

Our first observation follows immediately from the proof of Theorem \ref{edge col}.
\begin{proposition}\label{razy 3}
If each vertex of a graph $G$ has degree divisible by $3$, then Maj'$(G)\le 4$.
\end{proposition}

In the next proofs, we make use of two results involving the following concept of the $k$-discrepancy of a graph $G$. Given a set $C$ of $k$ colors, an edge coloring $c:E(G)\rightarrow C$, a vertex $v$ of $G$, and a color $i\in C$ let $d_{E_i}(V)$ denote the number of edges incident to $v$ colored with $i$. By the $k$-{\it discrepancy} ${\mathcal D}_k(G)$ of a graph $G$ we mean the following minimum over all $k$-edge-colorings $c$ of $G$:
$${\mathcal D}_k(G)=\min_c\; \max_{v\in V(G)}\; \max_{i,j\in C}\;|d_{E_i}(v)-d_{E_j}(v)|.$$

In 1975, de Werra proved in \cite{deW75} (for a simpler proof, consult \cite{HiStSl98}) the following. 
\begin{theorem} [\cite{deW75}]\label{de}
If $G$ is a bipartite graph, then  ${\mathcal D}_k(G)\le 1$ for any $k\ge 1$.  
\end{theorem}

Recently, P\k{e}ka{\l}a and Przyby{\l}o  \cite{PePr25} proved the following general bound for the $k$-discrepancy of graphs.
\begin{theorem}[\cite{PePr25}]\label{PP}
${\mathcal D}_k(G)\le 2$ for any graph $G$ and any $k\ge 1.$
\end{theorem}

Now, we state our most general result regarding Conjecture \ref{Maj'4}. 
\begin{theorem}\label{delta7}
If $G$ is a graph with minimum degree $\delta(G)\ge 7$, then $${\mathrm Maj'}(G)\le 4.$$
\end{theorem}
\begin{proof}
By Theorem \ref{PP} for $k=4$, any graph $G$ admits an edge 4-coloring $c$ such that $|d_{E_i}(v)-d_{E_j}(v)|\le 2$ for $1\le i<j\le 4.$

Let $e=v_1v_2\in E(G)$. For $i=1,2$, let $k_i$ be the size of the largest color class among the edges incident to $v_i$. It follows that $4k_i-6\le d(v_i), i=1,2.$ Then the number of edges adjacent to $e$ that have the same color is at most 
$$k_1+k_2\le \left\lfloor\frac14(d(v_1)+d(v_2))+3\right\rfloor.$$  The edge $e$ has $d(v_1)+d(v_2)-2$ adjacent edges. Hence, we have to show that the inequality  
\begin{equation}\label{pp}
\left\lfloor\frac14(d(v_1)+d(v_2))+3\right\rfloor\le \left\lfloor\frac12(d(v_1)+d(v_2)-2)\right\rfloor
\end{equation}
holds whenever $\delta(G)\ge 7.$ If we omit the floors, then we get the inequality $d(v_1)+d(v_2)\le 16.$ Hence, the claim holds when $\delta(G)\le 8$. One can easily verify that inequality (\ref{pp}) is also true if $d(v_1), d(v_2)\ge 7,$ even though $d(v_1)+d(v_2)\le 15.$
\end{proof}

Let us admit that analogous arguments justifies the following.

\begin{proposition}\label{delta9}
If $G$ is a graph with minimum degree $\delta(G)\ge 9$, then $${\mathrm Maj'}(G)\le 3.$$ 
\end{proposition}

For bipartite graphs, we can lower the bound for the minimum degree of $G$.
\begin{proposition}
If $G$ is a bipartite graph with minimum degree $\delta(G)\ge 4$, then  $${\mathrm Maj'}(G)\le 4.$$   
\end{proposition}
\begin{proof}  We apply Theorem \ref{de} of de Werra for $k=4$ and argue similarly as in the proof of Theorem \ref{delta7}.

Let $e=v_1v_2\in E(G)$. For $i=1,2$, let $k_i$ be the size of the largest color class of the edges incident to $v_i$. It follows from  Theorem \ref{de} that $4k_i-3\le d(v_i), i=1,2.$ Then the number of edges adjacent to $e$ that have the same color is at most $$k_1+k_2\le \left\lfloor\frac14(d(v_1)+d(v_2))+\frac32\right\rfloor.$$ 
Now, it suffices to show that the inequality 
\begin{equation}\label{bip}
\left\lfloor\frac14(d(v_1)+d(v_2))+\frac32\right\rfloor\le \left\lfloor\frac12(d(v_1)+d(v_2)-2)\right\rfloor
\end{equation}
is true whenever $d(x_1)+d(x_2)\le 20$. By omitting the floors, we get the inequality $d(v_1)+d(v_2)\ge 14$.
This is obviously the case whenever $\delta(G)\ge 7.$ It is easy to check that the inequality (\ref{bip}) is maintained in all cases when $d(v_1),d(v_2)\ge 4$ and $d(v_1)+d(v_2)\le 13$. 
\end{proof}

Using the same arguments, one can show that Maj'$(G)\le 3$ for every bipartite graph $G$ with $\delta(G)\ge 5$.

\vspace{3mm} 
Conjecture \ref{Maj'4} also holds for Eulerian graphs. 
\begin{proposition}\label{Euler}
If all vertices of a graph $G$ have even degrees, then $${\mathrm Maj'}(G)\le 4.$$    
Moreover, if $G$ is connected and the size of $G$ satisfies $\|G\|\equiv 0\mod 3$, then Maj'$(G)\le 3$.
\end{proposition}
\begin{proof}
Clearly, it suffices to prove the claim for any connected graph $G$. Let $m=\|G\|$ and let $W=e_1,\ldots,e_m,e_1$ be an Euler tour in $G$. We color the subsequent edges of $W$, starting from $e_1$ with three colors in the following way: $1,2,3,1,2,3,\ldots,1,2,3$, as long as possible. That is, we color the edges $e_1,\ldots,e_p$, where $p=3\lfloor m/3\rfloor$. 
Observe that each edge $e_i$, possibly except $i\in\{1,p\}$ if $p\ne m$, is majority colored. Indeed, the edges $e_{i-1}$ and $e_{i+1}$ get two different colors, and every passage of $W$ through an end-vertex of $e_i$ adds two different colors of edges incident to it. Consequently, we obtain a strong majority edge-coloring of $G$ if $m\equiv 0\mod 3$. 

Otherwise, if $m\equiv 1,2\mod 3$, we put color 4 on the remaining one or two edges of $W$. 
\end{proof}
\vspace{3mm}

 Let $G$ be an $r$-regular graph. It follows from Proposition \ref{delta7} and Observation \ref{cycle2} that Maj'$(G)\le 3$ when $r\ge 9$ and $r=2$, respectively. Moreover, Maj'$(G)\le 4$ for $r\in\{3,6\}$ by Proposition \ref{razy 3}, and for $r\in\{4,8\}$ by Proposition \ref{Euler}. This bound can be improved for $r=6$.
\begin{observation}\label{6r}
${\mathrm Maj}'(G)\le 3$ for every $6$-regular graph $G$.
\end{observation}
\begin{proof}
By the well-known Petersen theorem, every 6-regular graph $G$  admits a decomposition into three 2-factors. Color each 2-factor with a distinct color. Thus, every edge $e$ is adjacent to two edges of the same color as the color of $e$, as well as to four edges of each of the other two colors. The number of edges adjacent to $e$ is equal to 10. Therefore, this is a strong majority edge-coloring of $G$.   
\end{proof}

\vspace{4mm}
\noindent
{\bf Acknowledgment}
\vspace{2mm}

\noindent
The authors thank Zden\v ek Dvo\v r\'ak for his improvement of the proof of Theorem \ref{2Delta+1}.

\bibliographystyle{abbrv} 
\bibliography{literat.bib}

@Article{BKPPRW23,
 Author = {Felix {Bock} and Rafa{\l} {Kalinowski} and Johannes {Pardey} and Monika {Pil\'sniak} and Dieter {Rautenbach} 
 and Mariusz {Wo\'zniak}},
 Title = {{Majority Edge-Colorings of Graphs.}},
 Journal = {{Electron. J. Combin.}},
 Volume = {30},
 Number = {},
 Pages = {\#P1.42},
 Year = {2023}
}

@Article{Lov66,
 Author = {L\'aszl\'o  {Lov\'asz}},
 Title = {{On decomposition of graphs.}},
 Journal = {{Stud. Sci. Math. Hungar.}},
 Volume = {1},
 Number = {},
 Pages = {237--238},
 Year = {1966}
}

@Article{ShMi90,
 Author = {Saharon {Shelah} and E.C. {Milner}},
 Title = {{Graphs with no unfriendly partitions.}},
 Journal = {{ in A. Baker, B. Bollob\'as, A. Hajnal, (eds.), A Tribute to Paul Erdős, Cambridge University Press}},
 Volume = {},
 Number = {},
 Pages = {373–384},
 Year = {1990}
}

@Article{BDGS10,
 Author = {Henning {Bruhn} and Reinhard {Diestel} and Agelos {Georgakopoulos} and Philipp {Spr\"ussel}},
 Title = {{Every rayless graph has an unfriendly partition.}},
 Journal = {{Combinatorica}},
 Volume = {30},
 Number = {5},
 Pages = {521--532},
 Year = {2010}
}

@Article{Ber17,
 Author = {Eli {Berger}},
 Title = {{Unfriendly partitions for graphs not containing a subdivision of an infinite clique.}},
 Journal = {{Combinatorica}},
 Volume = {37},
 Number = {2},
 Pages = {157--166},
 Year = {2017}
}

@Article{KOSvW17,
 Author = {Stephan {Kreutzer} and Sangil {Oum} and Paul {Seymour} and Dominic {van der Zyphen} and David R. {Wood}},
 Title = {{Majority colourings of digraphs.}},
 Journal = {{Electron. J. Combin.}},
 Volume = {24},
 Number = {},
 Pages = {\#P2.25},
 Year = {2017}
}

@Article{AMP90,
 Author = {R. {Aharoni} and E.C. {Milner} and K. {Prikry}},
 Title = {{Unfriendly partitions of a graph}},
 Journal = {{J. Combin. Theory Ser. B}},
 Volume = {50},
 Number = {1},
 Pages = {1--10},
 Year = {1990}
}

@Book{Diestel,
Author = {Reinhard {Diestel} }, 
Title = {{Graph Theory, 5th Edition.}}, 
Publisher = {{Springer-Verlag Berlin Heidelberg New York}},
Year = {2016}}

@Article{KPS25,
 Author = {Rafa{\l} {Kalinowski} and Monika {Pil\'sniak} and Marcin {Stawiski}},
 Title = {{Unfriendly Partition Conjecture holds for line graphs.}},
 Journal = {{Combinatorica}},
 Volume = {45},
 Number = {1},
 Pages = {art. no. 3},
 Year = {2025}
}

@Article{PePr25,
 Author = {Pawe{\l} {P\k{e}ka{\l}a} and Jakub {Przyby{\l}o}},
 Title = {{On List Extensions of the Majority Edge Colourings.}},
Journal = {{Electron. J. Combin.}},
 Volume = {32},
 Number = {4},
 Pages = {\#P4.38},
 Year = {2025}
}

@Article{HiStSl98,
title = {A Vertex-Splitting Lemma, de Werra's Theorem, and Improper List Colourings},
journal = {J. Combin. Theory, Ser. B},
volume = {72},
number = {1},
pages = {91-103},
year = {1998},
issn = {0095-8956},
doi = {https://doi.org/10.1006/jctb.1997.1793},
url = {https://www.sciencedirect.com/science/article/pii/S0095895697917937},
author = {A.J.W. Hilton and D.S.G. Stirling and T. Slivnik}}

@Article{deW75,
 Author = {D. de Werra},
 Title = {A few remarks about chromatic scheduling},
 Journal = {Combinatorial Programming:
Methods and Applications'' (B. Roy, Ed.)},
Publisher = {Reidel, Dordrecht},
Pages = {337--342},
 Year = {1975}}

@article{ABGGPZ25,
title = {Mrs. Correct and majority colorings},
journal = {Discrete Math.},
volume = {348},
number = {11},
pages = {114577},
year = {2025},
issn = {0012-365X},
doi = {https://doi.org/10.1016/j.disc.2025.114577},
url = {https://www.sciencedirect.com/science/article/pii/S0012365X25001852},
author = {Marcin Anholcer and Bartłomiej Bosek and Jarosław Grytczuk and Grzegorz Gutowski and Jakub Przybyło and 
Mariusz Zając}}

\end{document}